\documentclass[10pt,a4paper]{amsart}

\pdfoutput=1
\usepackage[english]{babel}
\usepackage{amssymb,amsmath,latexsym,amsthm}
\usepackage{mathrsfs}
\usepackage{xcolor}
\usepackage[
pagebackref=true,
pdfauthor={Olga Balkanova and Dimitrios Chatzakos and Giacomo Cherubini and Dmitry Frolenkov and Niko Laaksonen},
pdftitle={Prime Geodesic Theorem in the Three-dimensional Hyperbolic Space},
pdfkeywords={prime geodesic},
pdfcreator={Pdflatex},
pdfpagemode  = UseNone,
pdfstartview = FitH
]{hyperref}
\hypersetup{colorlinks=true}
\usepackage{todonotes}

\numberwithin{equation}{section}

\theoremstyle{plain}
\newtheorem{thm}{Theorem}[section]
\newtheorem{lemma}[thm]{Lemma}
\newtheorem{corollary}[thm]{Corollary}

\theoremstyle{definition}
\newtheorem{rmk}[thm]{Remark}

\newcommand{\reals}{\mathbb{R}}

\newcommand{\wbar}{\overline{w}}
\newcommand{\zin}{\mathbb{Z}[i]\setminus\{0\}}

\DeclareMathOperator*{\vol}{vol}

\DeclareMathOperator*{\Li}{Li}

\usepackage{verbatim}
\renewcommand\MR[1]{}

\title[Prime Geodesic Theorem in the Hyperbolic 3-space]{Prime Geodesic Theorem in the 3-dimensional Hyperbolic Space}

\author[Balkanova]{Olga Balkanova}
\address{
	Department of Mathematical Sciences,
	Chalmers University of Technology
	and
	University of Gothenburg,
	Chalmers tv\"argata 3, 412 96 Gothenburg, Sweden
}
\email{olgabalkanova@gmail.com}

\author[Chatzakos]{Dimitrios Chatzakos}
\address{
    Universit\'e de Lille 1 Sciences et Technologies
    and
    Centre Européen pour les Mathématiques, la Physique et leurs interactions (CEMPI),
    Cit\'e Scientifique, 59655 Villeneuve d’ Ascq C\'edex, France
}
\email{Dimitrios.Chatzakos@math.univ-lille1.fr}

\author[Cherubini]{Giacomo Cherubini}
\address{
    Dipartimento di Matematica,
    Universit\`a degli Studi di Genova,
    via Dodecaneso 35, 16146, Genoa, Italy
}
\email{cherubini@dima.unige.it}

\author[Frolenkov]{Dmitry Frolenkov}
\address{
	National Research University Higher School of Economics, Moscow, Russia
	and
	Steklov Mathematical Institute of Russian Academy of Sciences, 8 Gubkina st., Moscow, 119991, Russia
}
\email{frolenkov@mi.ras.ru}

\author[Laaksonen]{Niko Laaksonen}
\address{
    McGill University, Department of Mathematics and Statistics,
    Burnside Hall, 805 Sherbrooke Street West, Montreal, Quebec, H3A 0B9
    Canada
}
\email{n.laaksonen@ucl.ac.uk}

\thanks{
	The first author 
	was partially supported by the Royal Swedish Academy of Sciences project no. MG2018-0002
	and by the European Research Council under the European Community's Seventh Framework Programme
	(FP7/2007-2013) with ERC Grant Agreement nr. 615722 MOTMELSUM.
	The second author 
	would like to thank the School of Mathematics  of
    University of Bristol and the Mathematics department of King's College
    London for their support and hospitality during the academic year 2016-17.
    During these visits, he received funding from an LMS 150th Anniversary
    Postdoctoral Mobility Grant 2016-17 and the European Union’s Seventh
    Framework Programme (FP7/2007-2013) / ERC grant agreement no. 335141 \textit{Nodal}.
    He is also currently supported by the Labex CEMPI (ANR-11-LABX-0007-01).
	The third author 
	was supported by a Leibniz fellowship, and thanks the MFO
    for excellent working conditions, and by a ``Ing. Giorgio Schirillo''
    postdoctoral grant by the ``Istituto Nazionale di Alta Matematica''.
    The last author 
	would like to thank the Mathematics department at KTH in Stockholm, where he received
    funding from KAW 2013.0327. He would also like to thank the Department of
    Mathematics and Statistics at McGill University
    and the Centre de Recherches Math\'ematiques for their
    hospitality.
    We also want to thank Yiannis Petridis, Maksym
    Radziwill and Kannan Soundararajan for their helpful comments.
}



\date{\today}

\begin{document}


\begin{abstract}
    For $\Gamma$ a cofinite Kleinian group acting on $\mathbb{H}^3$, we study the
    Prime Geodesic Theorem on $M=\Gamma \backslash \mathbb{H}^3$, which asks about
    the asymptotic behaviour of lengths of primitive closed geodesics (prime
    geodesics) on $M$.
    Let $E_{\Gamma}(X)$ be the error in the counting of prime geodesics with length
    at most $\log X$.
    For the Picard manifold, $\Gamma=\mathrm{PSL}(2,\mathbb{Z}[i])$,
    we improve the classical bound of Sarnak, $E_{\Gamma}(X)=O(X^{5/3+\epsilon})$,
    to $E_{\Gamma}(X)=O(X^{13/8+\epsilon})$.
	In the process we obtain
    a mean subconvexity estimate for the Rankin--Selberg $L$-function attached to
    Maass--Hecke cusp forms.
    We also investigate the second moment of $E_{\Gamma}(X)$ for a general cofinite group $\Gamma$,
    and show that it is bounded by $O(X^{16/5+\epsilon})$.
\end{abstract}

\maketitle

\section{Introduction}

Prime Geodesic Theorems describe the asymptotic behaviour
of primitive closed geodesics on hyperbolic manifolds.
The classical case is that of Riemann surfaces $\Gamma \backslash \mathbb{H}^2$,
where $\Gamma \subset \hbox{PSL}(2, {\mathbb R})$ is a cofinite Fuchsian group.
This problem was first studied by Huber~\cite{huber_zur_1961,huber_zur_1961-1}
and most importantly by Selberg
(see e.g.~\cite[Theorem~10.5]{iwaniec_spectral_2002})
as a consequence of his trace formula.
More specifically, let $\psi_{\Gamma}(X)$ denote
the analogue of the summatory von Mangoldt function, namely
\[
\psi_{\Gamma}(X) = \sum_{N(P) \leq X} \Lambda_{\Gamma} (N(P)),
\]
where $\Lambda_{\Gamma} (N(P)) = \log N(P_0)$, if $P$ is a power of the primitive
hyperbolic conjugacy class $P_0$, and zero otherwise.
Here $N(P)$ denotes the norm of $P$,
and the length of the closed geodesic corresponding to $P$ is $\log N(P)$.
Selberg proved that, as $X\to\infty$, we have
\[
\psi_{\Gamma}(X) = \sum_{1/2 < s_j \leq 1} \frac{X^{s_j}}{s_j} + E_{\Gamma}(X),
\]
where the \emph{full} main term is a finite sum that comes from the small eigenvalues
of the hyperbolic Laplacian, $\lambda_j = s_j(1 - s_j) < 1/4$,
and the error term $E_{\Gamma}(X)$ is bounded by $O(X^{3/4})$.

When $\Gamma$ is arithmetic,
further improvements on the bound for the error term were deduced
by Iwaniec~\cite[Theorem~2]{iwaniec_prime_1984}, Luo and Sarnak~\cite[Theorem~1.4]{luo_quantum_1995}
(see also Koyama~\cite{koyama_prime_1998}), and Cai~\cite{cai_prime_2002}.
The crucial step in all of these works is proving a non-trivial bound
on spectral exponential sums (see Section \ref{S03}).
Recently, for the modular group, Soundararajan and Young~\cite[Theorem~1.1]{soundararajan_prime_2013}
proved the currently best known result, which gives
\[
E_{\Gamma}(X) = O_{\epsilon}(X^{25/36 + \epsilon}).
\]
They do this by exploiting a connection to Dirichlet $L$-functions and using
an inequality of Conrey and Iwaniec~\cite[Corollary~1.5]{conrey_cubic_2000}
to estimate their value on the critical line.

In this paper we study the Prime Geodesic Theorem in the three-dimensional
hyperbolic space $\mathbb{H}^3$.
Let $\Gamma$ be a cofinite Kleinian group and let $\psi_\Gamma(X)$ be
the analogous summatory von Mangoldt function attached to $\Gamma$,
which counts hyperbolic (and loxodromic) conjugacy classes in the group.
The small eigenvalues $\lambda_j = s_j(2-s_j) < 1$ provide a finite number of terms
that form the full main term of $\psi_\Gamma(X)$, as $X\to\infty$, namely
\begin{equation*}
   M_{\Gamma}(X) := \sum_{ 1< s_j \leq 2} \frac{X^{s_j}}{s_j},
\end{equation*}
and we write
\[
E_{\Gamma}(X)=\psi_\Gamma(X)-M_{\Gamma}(X).
\]

For general cofinite groups, Sarnak~\cite[Theorem~5.1]{sarnak_arithmetic_1983} proved
the following nontrivial bound for the error term:
\begin{equation}\label{1509:eq001}
E_{\Gamma}(X) = O_{\epsilon}(X^{5/3+\epsilon}).
\end{equation}
For the Picard group, $\Gamma = \hbox{PSL}(2, \mathbb{Z}[i])$,
Koyama~\cite[Theorem~1.1]{koyama_prime_2001} improved this bound conditionally to
$O_{\epsilon}(X^{11/7+\epsilon})$ by assuming a mean Lindel\"of hypothesis for
symmetric square $L$-functions attached to Maass forms on $\Gamma \backslash \mathbb{H}^3$.
Our main result is the following unconditional improvement of~\eqref{1509:eq001}
in the arithmetic case $\Gamma=\mathrm{PSL}(2,\mathbb{Z}[i])$.
\begin{thm}\label{intro:thm01}
Let $\Gamma=\mathrm{PSL}(2,\mathbb{Z}[i])$. Then
\[E_{\Gamma}(X) = O_{\epsilon} (X^{13/8+\epsilon}).\]
\end{thm}
The error term $E_{\Gamma}$ can be understood as a summation over the non-trivial zeros
of the Selberg zeta function through an explicit formula
(see~\eqref{1610:eq001}). This is analogous to the relationship between the
explicit formula for the Riemann zeta function and the Prime Number Theorem,
and it is the starting point in our proof of Theorem~\ref{intro:thm01}.
The fundamental tool in our proof is the Kuznetsov formula (Theorem~\ref{Kuznetsovformula}),
which relates the problem to the study of Kloosterman sums for Gaussian integers.

In two dimensions Cherubini and Guerreiro examined the second moment of the error term
for general cofinite Fuchsian groups, and they improved the error term
on average \cite[Theorems~1.1,~1.4]{cherubini_mean_2017}.
We initiate the study of the second moment of $E_\Gamma(X)$ in three dimensions
and prove the following theorem by an application of the Selberg trace formula.
\begin{thm}\label{intro:thm02}
Let $\Gamma$ be a cofinite Kleinian group.
Let $V$, $\Delta$ be sufficiently large positive real numbers with $\Delta\leq V$.
Then
\[
\frac{1}{\Delta}\int_V^{V+\Delta} |E_{\Gamma}(x)|^2 dx \ll V^{18/5}\Delta^{-2/5}(\log V)^{2/5}.
\]
\end{thm}
By taking $\Delta=V$ in the above theorem, we see that $E_{\Gamma}(X)\ll
X^{8/5+\epsilon}$ on average. More precisely, we obtain the following result.
\begin{corollary}
For every $\eta>8/5$ there exists a set $A\subseteq [2,\infty)$ of finite
logarithmic measure \textup{(}i.e.~for which
$\int_A x^{-1}dx$ is finite\textup{)}, such that $E_\Gamma(X)=O(X^\eta)$ for $X\to\infty$, $X\notin A$.
\end{corollary}

\begin{rmk}
Using a mean-to-max argument, we recover
Sarnak's bound~\eqref{1509:eq001} as a corollary of Theorem~\ref{intro:thm02}.
Indeed, a bound of the type $O(V^\beta\Delta^{-\gamma})$ for the second moment of $E_\Gamma(X)$
in short intervals
leads to the estimate $E_\Gamma(X)=O(X^\alpha)$, where $\alpha=\frac{\beta+\gamma}{2+\gamma}$.
Theorem~\ref{intro:thm02} allows $\beta=18/5+\epsilon$ and $\gamma=2/5$,
leading to~\eqref{1509:eq001}.
\end{rmk}

\begin{rmk}\label{2911:rmk01}
    It is interesting to speculate on what is expected for the error term
    $E_{\Gamma}(X)$. In analogy with two dimensions, it is tempting to say that
    one should expect a bound of size the square root of the main term, namely $E_\Gamma(X)\ll X^{1+\epsilon}$.
    However, there are reasons indicating that it is not possible
    to reach this bound, the most relevant being that
    the explicit formula~\eqref{1610:eq001} in its current
    form has a natural barrier $E_{\Gamma}(X)\ll X^{3/2+\epsilon}$ by taking
    $T=X^{1/2}$. We discuss this issue further in Remark~\ref{s03:rmk001},
    after the definition of the explicit formula and of the
    spectral exponential sum associated to it.
\end{rmk}

\begin{rmk}
	By using a different method, the authors of~\cite{balkanova2018} proved
	an estimate $E(X)\ll X^{3/2+\theta/2}$, which explicitly depends on the
	subconvexity exponent $\theta$ for quadratic Dirichlet $L$-functions over $\mathbb{Z}[i]$.
\end{rmk}

Finally, we recall the connection between class numbers and the Prime Geodesic Theorem
for the Picard group.
Denote by $\mathscr{D}$ the set of discriminants of binary
quadratic forms (over $\mathbb{Z}[i]$), so that
\[
\mathscr{D}
=
\{
m \in \mathbb{Z}[i]: \;
m \equiv y^2(\bmod{4}) \text{ for some } y \in \mathbb{Z}[i]
\text{ and $m$ is not a perfect square}
\},
\]
and let $h(d)$ be the number of classes of primitive binary quadratic forms of discriminant $d$.
Moreover, for $d\in\mathscr{D}$, consider the Pell-type equation $t^2-du^2=4$. All the solutions are
generated by a fundamental unit
\[
\epsilon_d = \frac{t_0+\sqrt{d}u_0}{2}, \quad |\epsilon_d|>1.
\]
Sarnak~\cite[Corollary~4.1]{sarnak_arithmetic_1983} proved the identity
\[
\sum_{\substack{d\in\mathscr{D}\\ |\epsilon_d| \leq x}} h(d)\log|\epsilon_d|^2
=
\pi_\Gamma(x^2) = \#\{P_0 \text{ primitive hyperbolic, } N(P_0)\leq x^2\},
\]
which we can relate to $\psi_\Gamma$ via summation by parts.
Thus every result on $\psi_\Gamma(x)$ translates immediately to a statement on the
average size of the class numbers $h(d)$ (see~\cite[Theorem~7.1]{sarnak_arithmetic_1983}).
Therefore, we obtain the following corollary as a consequence of
Theorem~\ref{intro:thm01} and Theorem~\ref{intro:thm02}.

\begin{corollary} \label{corollary2}
For every $\epsilon>0$ we have
\begin{equation}\label{2609:eq001}
\sum_{\substack{d\in\mathscr{D}\\ |\epsilon_d| \leq X}} h(d) = \Li(X^4) +
O_\epsilon(X^{13/4+\epsilon}).
\end{equation}
On average over $X$, we obtain, for $V\geq\Delta\gg 1$,
\begin{equation}\label{2609:eq002}
   \frac{1}{\Delta}\int_V^{V+\Delta}
   \bigg|\sum_{d\in\mathscr{D}\atop |\epsilon_d|
   \leq
   X} h(d)-\Li(X^4)\bigg|^2 dX
   \ll_{\epsilon}
   V^{18/5+\epsilon}\Delta^{-2/5}.
\end{equation}
\end{corollary}

The rest of this paper is organized as follows: in Section~\ref{S02} we discuss the necessary background material
and the main tools of our proofs, namely the Selberg and the Kuznetsov trace formulas,
whereas in Sections~\ref{S03} and~\ref{S04} we prove Theorem~\ref{intro:thm01} and
Theorem~\ref{intro:thm02}, respectively.

\section{Background  and auxiliary results}\label{S02}

The hyperbolic space $\mathbb{H}^{3}$ can be described as the set of points
$p=z+jy=(x_{1}, x_{2}, y)$, where $z=x_{1}+ix_{2} \in \mathbb{C}$ and $y>0$.
This space is endowed with the hyperbolic metric, whose line element is $ds^2=y^{-2}(dx_1^2+dx_2^2+dy^2)$.
If we represent $p$ as a quaternion whose fourth component is zero,
then given a matrix $M=\left(\begin{smallmatrix}a&b\\c&d\end{smallmatrix}\right)$
in $\mathrm{PSL}(2,\mathbb{C})$ we have the orientation-preserving isometric action
\[Mp := (ap+b)(cp+d)^{-1}.\]
A discrete group $\Gamma\leq \mathrm{PSL}(2,\mathbb{C})$
is said to be \emph{cofinite} if the quotient $\Gamma\backslash\mathbb{H}^3$ has finite volume
(with respect to the metric induced on the quotient space). It is called \emph{cocompact} if the quotient
is a compact space.

Let us now fix a cofinite group $\Gamma$ and consider the
Laplace--Beltrami operator $\Delta$ acting on $L^2(\Gamma\backslash\mathbb{H}^3)$.
This admits eigenvalues
\[\lambda_j=s_j(2-s_j),\quad s_{j}=1+ir_j,\]
where $r_j\in\reals$ or $r_j$ purely imaginary in the interval $(0,i]$.
If $\Gamma\backslash\mathbb{H}^{3}$ has cusps, then there is also a continuous spectrum spanning $[1,\infty)$.

We will need the Fourier expansion of the cusp form $u_j$ attached to
$\lambda_j$, which reads~\cite[\S3~Theorem~3.1]{elstrodt_groups_1998}
\[
u_j(p)=
y \sum_{0\neq n\in\Gamma^\sharp} \rho_j(n)K_{ir_j}(2\pi|n|y)e(\langle n, z\rangle),
\]
where $\Gamma^\sharp$ is the dual lattice of $\Gamma$, $e(z)=\exp(2\pi iz)$,
and $\langle x,y\rangle$ is the standard inner product on $\reals^2\cong \mathbb{C}$.

\subsection{Weyl law for Kleinian groups}
It is a classical problem to understand the distribution of the discrete
spectrum of the Laplace operator on hyperbolic manifolds of finite volume.
In fact, apart from the eigenvalue $\lambda_0=0$, which always occurs,
for a general cofinite group we do not even know if there
are any other eigenvalues at all (see~\cite{phillips_weyl_1985}
for a discussion on this topic in the two-dimensional case).
The Weyl law describes the asymptotic behaviour
of both the discrete and continuous spectrum in an expanding window.
More precisely it states, in our situation, that \cite[\S6~Theorem~5.4]{elstrodt_groups_1998}
\[
\#\{r_j\leq T\} -
\frac{1}{4\pi}\int_{-T}^{T}\frac{\varphi'}{\varphi}(1+ir)\,dr
\sim
\frac{\vol(\Gamma\backslash\mathbb{H}^3)}{6\pi^{2}} T^3,
\]
as $T\to\infty$. Note that for cocompact groups the second term vanishes, and we deduce
that there are infinitely many eigenvalues in this case.
For our purposes we will need a control on the size of the spectrum (both discrete and continuous)
in windows of unit length. To this end, we appeal to a result of
Bonthonneau~\cite[Theorem~2]{bonthonneau_weyl_2015}, who gives
a Weyl law for hyperbolic manifolds with good error terms. In our case his theorem simplifies to the following.
\begin{thm}
    Let $\Gamma$ be a cofinite Kleinian group. Then
    \[
    \#\{r_{j}\leq T\} -
    \frac{1}{4\pi}\int_{-T}^{T}\frac{\varphi'}{\varphi}(1+ir)\,dr
    =
    \frac{\vol(\Gamma\backslash\mathbb{H}^{3})}{6\pi^{2}}T^{3}+O\left(\frac{T^{2}}{\log T}\right).
    \]
\end{thm}
From the Maass--Selberg relations~\cite[\S3~Theorem~3.6]{elstrodt_groups_1998} we obtain
the following upper bound on unit intervals:
\begin{equation}\label{weyl-law}
   \#\{r_{j}\in[T,T+1]\}
   +
   \int_{T\leq|r|\leq T+1}\left\lvert\frac{\varphi'}{\varphi}(1+ir)\right\rvert dr
   \ll T^{2}.
\end{equation}

\subsection{The Selberg trace formula}
For general cofinite groups, the Selberg trace formula is perhaps one of the most
effective tools available to attack problems in the spectral theory of automorphic forms.
The formula relates geometric information attached to a group
to spectral data of the hyperbolic Laplacian.
In preparation to stating the formula rigorously, we give a few definitions.

Let $M\in\mathrm{PSL}(2,\mathbb{C})$, $M\neq I$, and consider its trace $\mathrm{tr}(M)$.
If $\mathrm{tr}(M)$ is not real, then $M$ is called \emph{loxodromic}.
Otherwise, depending on whether the absolute value of the trace is smaller,
equal or larger than $2$,
$M$ is called \emph{elliptic}, \emph{parabolic} or \emph{hyperbolic}, respectively.
Every hyperbolic or loxodromic element $M$ is conjugate in $\mathrm{PSL}(2,\mathbb{C})$
to a unique element
\[
\begin{pmatrix}
a(M) & 0 \\
0    & a(M)^{-1}
\end{pmatrix}
\]
with $|a(M)|>1$. The quantity $N(M):=|a(M)|^2$ is called the \emph{norm} of $M$.
Since this is invariant under conjugation, we define the norm of a conjugacy class
to be the norm of any of its representatives.
Finally, for $\Gamma$ a discrete subgroup of $\mathrm{PSL}(2,\mathbb{C})$
and $\gamma\in\Gamma$, we say that $\gamma$ is \emph{primitive}
if it has minimal norm among the elements of the centralizer $C(\gamma)$ in $\Gamma$.
Since the notion is invariant under conjugation, we can extend it to conjugacy classes in $\Gamma$.

We are now ready to state the Selberg trace formula.
For simplicity, we assume that we only have one cusp at infinity.
\begin{thm}[Selberg trace formula {\cite{elstrodt_groups_1998,tanigawa_selberg_1977}}]\label{s02:selberg-trace-formula}
    Let $h$ be an even function,
    holomorphic in $|\Im r| < 1+\epsilon_{0}$ for some $\epsilon_{0}>0$,
    and assume that $h(r)=O((1+|r|)^{-3-\epsilon})$ in the strip.
    Furthermore, let $g$ be the Fourier transform of $h$,
    defined with the convention
   \begin{equation}\label{2909:eq001}
       g(x) = \frac{1}{2\pi}\int_{-\infty}^{\infty}h(r)e^{-irx}dr.
   \end{equation}
   Then
   \[
   \sum_{j} h(r_{j})
   -
   \frac{1}{4\pi}\int_{-\infty}^{\infty} h(r)\frac{\varphi'}{\varphi}(1+ir)dr
   =
   I + E + H + P,
   \]
   where
    \begin{align*}
        I &= \frac{\vol(\Gamma\backslash\mathbb{H}^{3})}{4\pi^{2}}\int_{-\infty}^{\infty}h(r)r^{2}dr,
		\\
        E &= \frac{1}{4}\sum_{\{R\} \mathrm{nce}} \frac{g(0)\log
            N(T_{0})}{|\mathcal{E}(R)|\left(\sin\frac{\pi k}{m(R)}\right)^{2}},
		\\
		H &= \sum_{\{T\} \mathrm{lox}} \frac{g(\log N(T)) \Lambda_{\Gamma}(
            N(T))}{|\mathcal{E}(T)| |a(T)-a(T)^{-1}|^{2}},
	\end{align*}
	\begin{align*}
        P &=
        \frac{g(0)}{[\Gamma_{\infty}:\Gamma'_{\infty}]}\left(\frac{\kappa_{\Lambda_{\infty}}}{2}-\gamma\right)
           + \frac{h(0)}{4}\left([\Gamma_{\infty}:\Gamma'_{\infty}]^{-1}-\varphi(0)\right)\\
        &\phantom{=}
         -\frac{1}{2\pi[\Gamma_{\infty}:\Gamma_{\infty}']}\int_{-\infty}^{\infty}h(r)\frac{\Gamma'}{\Gamma}(1+ir)dr+cg(0).
   \end{align*}
   The sum in $E$ runs over the elliptic conjugacy classes not stabilising the cusp
   at $\infty$ and the sum in $H$ runs over all loxodromic \textup{(}and
   hyperbolic\textup{)}
   conjugacy classes of~$\Gamma$.
\end{thm}
The function $\Lambda_{\Gamma}(N(T))$
is defined in the introduction,
and for the definition of the other quantities appearing in the theorem
we refer to~\cite{elstrodt_groups_1998}.
Under the hypotheses for $h$, all sums are absolutely convergent,
and we have $|\mathcal{E}(T)|=1$, apart from finitely many classes.
We note that there is a missing factor of $1/4\pi$ in front of the sums over
elliptic and hyperbolic conjugacy classes
in~\cite[p.~297]{elstrodt_groups_1998}.

\subsection{The Kuznetsov trace formula}
For arithmetic groups, such as $\mathrm{PSL}(2,\mathbb{Z}[i])$ or subgroups of it,
the Kuznetsov trace formula allows one to prove finer results than those
obtained solely by means of the Selberg trace formula.
The formula relates spectral data, namely Fourier coefficients of cusp forms,
to Kloosterman sums. Indirectly, these sums encode arithmetic information of the group,
and have been studied intensively in the case of the modular group
(see e.g.~\cite{iwaniec_fourier_1980,iwaniec_mean_1982}).

For Gaussian integers, Kloosterman sums are defined as follows.
Let $m,n,c\in\mathbb{Z}[i]$, with $c\neq 0$. Then
\[
S(m,n;c) := \sum_{ a \in (\mathbb{Z}[i]/(c))^{\times}} e(\langle m, a/c\rangle)
e(\langle n, a^*/c\rangle),
\]
where $a^*$ denotes the inverse of $a$ modulo the ideal $(c)$, that is
$a a^* \equiv 1\bmod{c}$. The Weil bound for these sums is
\cite[(3.5)]{motohashi_trace_1997}
\begin{equation}\label{s02:weil}
   |S(m,n;c)| \ll N(c)^{1/2} |(m,n,c)|  d(c),
\end{equation}
where $d$ is the number of divisors of $c$. When using the Weil bound it is also useful to recall that
the gcd, $|(m,n,c)|$, is essentially one on average over $c$, since
\begin{equation}\label{0608:eq001}
\sum_{N(c)\leq x} |(m,n,c)| \ll x^{1+\epsilon} |mn|^{\epsilon}.
\end{equation}

\begin{thm}[Kuznetsov formula for ${\hbox{PSL}(2,{\mathbb Z}[i])} \backslash \mathbb{H}^3$
\cite{motohashi_trace_1996,motohashi_trace_1997}]\label{Kuznetsovformula}
Let $h$ be an even function,
holomorphic in $|\Im r|<1/2+\epsilon$, for some $\epsilon>0$,
and assume that $h(r)=O((1+|r|)^{-3-\epsilon})$ in the strip.
Then, for any non-zero $m,n \in {\mathbb Z}[i]$:
\[
D + C = U + S,
\]
with
\[
\begin{split}
D &= \sum_{j=1}^{\infty}   \frac{ r_j \rho_j(n) \overline{\rho_j(m)} }{\sinh \pi r_j}  h(r_j),  \\
C &=  2 \pi \int_{-\infty}^{\infty}     \frac{\sigma_{ir}(n)
      \sigma_{ir}(m)}{ |mn|^{ir} |\zeta_{K}(1+ir)|^2}dr, \\
U &= \frac{\delta_{m,n} + \delta_{m,-n}}{\pi^2} \int_{-\infty}^{\infty} r^2 h(r) dr, \\
S &= \sum_{c \in\zin} \frac{S(m,n;c)}{|c|^2}
\int_{-\infty}^{\infty}    \frac{ir^2}{\sinh \pi r} h(r) H_{ir} \left(\frac{2 \pi \sqrt{\overline{mn}}}{c}\right) dr,
\end{split}
\]
where $\sigma_s(n) = \sum_{d|n} N(d)^s$ is the divisor function,
\[
H_{\nu} (z) = 2^{-2\nu} |z|^{2 \nu} J_{\nu}^*(z) J_{\nu}^{*}(\overline{z}),
\]
$J_{\nu}$ is the $J$-Bessel function of order $\nu$,
and $ J_{\nu}^*(z) = J_{\nu}(z) (z/2)^{-\nu}$.
\end{thm}

\section{Pointwise Bounds}\label{S03}

In this section we prove Theorem~\ref{intro:thm01}.
To begin with, we discuss the relationship between $E_\Gamma(X)$ and certain spectral exponential sums.
Let $\Gamma=\mathrm{PSL}(2,\mathbb{Z}[i])$.
Nakasuji~\cite[Theorem~4.1]{nakasuji_prime_2001} gives an explicit formula
that connects $E_\Gamma(X)$ to the spectral parameters of $\Gamma$, which
reads
\begin{equation}\label{1610:eq001}
    E_\Gamma(X)
    =
    2\Re\left(\sum_{0<r_j\leq T}\frac{X^{1+ir_j}}{1+ir_j}\right) + O\left(\frac{X^2}{T}\log X\right),
\end{equation}
for $1\leq T \leq X^{1/2}$, where $\lambda_j=1+r_j^2$.
Her result is actually valid for all cocompact and Bianchi groups.
A similar formula was proved by Iwaniec~\cite[Lemma~1]{iwaniec_prime_1984} for the case of the modular group,
and it provides an effective way to pass from the geometric quantity $E_\Gamma(X)$ to spectral data.

From~\eqref{1610:eq001} it is clear that studying $E_\Gamma(X)$ is
related to understanding the spectral exponential sum
\[
S(T,X) := \sum_{0<r_j\leq T} X^{ir_j}.
\]
\begin{rmk}\label{s03:rmk001}
Notice that the trivial bound for $S(T,X)$ is $O(T^{3})$ by Weyl's law. If we use
this in~\eqref{1610:eq001}, we recover Sarnak's bound $E_{\Gamma}(X)\ll
X^{5/3+\epsilon}$.
Assuming that we can take $T$ as large as possible in~\eqref{1610:eq001}, i.e.~$T=X^{1/2}$,
then the error is $O(X^{3/2+\epsilon})$, and the sum is bounded by the same quantity
provided that $S(T,X)\ll T^{2+\epsilon}X^\epsilon$. This bound on $S(T,X)$ is
also the strongest we can hope for with the method we use due to the error introduced
by the smoothing.

Now, let us ignore the limitation $T\leq X^{1/2}$ for a moment,
and go further, to see what type of bounds on $E_\Gamma(X)$ one can obtain from bounds on $S(T,X)$.
If we suppose that we have square root cancellation for $S(T,X)$, i.e.~$S(T,X)\ll T^{3/2+\epsilon}X^\epsilon$,
then by taking $T=X^{2/3}$ one gets $E_\Gamma(X)\ll X^{4/3+\epsilon}$.
This is still far from the exponent $1+\epsilon$
(square root bound for $E_\Gamma(X)$)
which was mentioned in Remark~\ref{2911:rmk01} in the introduction.
In fact, by using the explicit formula and a bound for $S(T,X)$, then the exponent $1+\epsilon$
can only be reached by assuming the extremely strong estimate $S(T,X)\ll
T^{1+\epsilon}X^\epsilon$, which seems unlikely to be true.
This kind of asymmetry between square root bounds for $S(T,X)$ and $E_\Gamma(X)$
does not occur in two dimensions, see~\cite[\S1]{iwaniec_prime_1984}, \cite[\S2]{petridis_local_2014}.
\end{rmk}

In this section we prove the following non-trivial bound on $S(T,X)$.
\begin{thm}\label{1610:thm01}
Let $\Gamma=\mathrm{PSL}(2,\mathbb{Z}[i])$, and let $X,T>2$. Then
\[
    S(T,X)\ll T^{2+\epsilon} X^{1/4+\epsilon}.
\]
\end{thm}
Theorem~\ref{intro:thm01} will follow in a straightforward manner from
Theorem~\ref{1610:thm01}.
To prove Theorem~\ref{1610:thm01} we will make use of the Kuznetsov formula (Theorem~\ref{Kuznetsovformula}),
and we will need an estimate for the Fourier coefficients $\rho_j(n)$ of
Maass--Hecke cusp forms.
The core of this section is devoted to proving such an estimate.
Let $L(s,u_j\otimes u_j)$ be the Rankin-Selberg convolution
\[
L(s,u_j\otimes u_j) := \sum_{n\in \zin}
\frac{|\rho_j(n)|^2}{N(n)^s},
\]
attached to the Maass--Hecke cusp form $u_{j}$.
We need an estimate on the mean value of $L(s,u_{j}\otimes
u_{j})$ in the spectral aspect on the critical line.
To do this we will first work with the symmetric square
$L$-function.
This is defined as
\begin{equation}\label{eq:defsym}
    L^{(2)}(s,u_{j}):=\frac{\zeta_{K}(2s)}{\zeta_{K}(s)}L(s,u_{j}\otimes
    u_{j})\frac{r_{j}}{\sinh\pi r_{j}}|v_{j}(1)|^{-2}.
\end{equation}
Here we use the standard notation
\[
\rho_j(n) = \sqrt{\frac{\sinh \pi r_j}{r_j}} v_j(n),
\]
and we recall the relation to the Hecke eigenvalues
$v_{j}(n)=v_{j}(1)\lambda_{j}(n)$.
It follows that the Dirichlet series of $L^{(2)}(s,u_{j})$ is given by
\[L^{(2)}(s,u_{j})=\sum_{n\in\zin}\frac{c_{j}(n)}{N(n)^{s}},\]
where the coefficients are defined as
$c_{j}(n)=\sum_{l^{2}k=n}{\lambda_{j}(k^{2})}$.
We prove the following estimate for the symmetric square $L$-function.
Let $r_{j}\sim T$ denote the interval $T<r_{j}\leq 2T$.
\begin{thm}\label{thm:subconvexitysym}
    Suppose $\Re w=\frac{1}{2}$ and let $u_{j}$ be a Maass--Hecke cusp form. We then have
    \[
	\sum_{r_{j}\sim T}|L^{(2)}(w,u_{j})|^{2} \ll |w|^{A} T^{4+\epsilon},
	\]
	for some positive constant $A$, and for arbitrarily small $\epsilon>0$.
\end{thm}
We can then use the Cauchy--Schwarz inequality,
the convexity bound $\zeta_{K}(w)\ll |w|^{1/2+\epsilon}$, and the upper
bound $|v_{j}(1)|\ll r_{j}^{\epsilon}$
(see~\cite{hoffstein_coefficients_1994} and~\cite[Proposition~3.1]{koyama_prime_2001})
to deduce the following corollary.
\begin{corollary}\label{cor:subconvexity}
    Let $u_{j}$ be a Maass--Hecke cusp form and suppose $\Re w=\frac{1}{2}$. Then
    \[
	\sum_{r_{j}\leq T}\frac{r_{j}}{\sinh\pi r_{j}}|L(w,u_{j}\otimes u_{j})|
	\ll
	|w|^{A} T^{7/2+\epsilon},
	\]
	for some positive constant $A$, and for arbitrarily small $\epsilon>0$.
\end{corollary}
The convexity bound in the spectral aspect would be $T^{4+\epsilon}$,
while Lindel\"of hypothesis would give $T^{3+\epsilon}$. Therefore
our theorem takes us halfway towards the goal.
In~\cite{koyama_prime_2001} Koyama assumes the above theorem with the
bound $O(|w|^A T^{3+\epsilon})$.
It should be noted that, following our proofs,
a bound of the form $O(|w|^A T^{3+\alpha})$ in Corollary~\ref{cor:subconvexity}, with $0<\alpha<1$,
would give the inequalities
\[
    S(T,X) \ll T^{(7+2\alpha)/4+\epsilon}X^{1/4+\epsilon} +T^{2},
\]
and
\[
E_\Gamma(X) \ll X^{\frac{11+4\alpha}{7+2\alpha}+\epsilon}.
\]
Therefore any improvement on Corollary~\ref{cor:subconvexity} would imply a refinement of our Theorem~\ref{intro:thm01}
and Theorem~\ref{1610:thm01}. Clearly, on taking $\alpha=0$ one recovers the conditional exponent $11/7$
of Koyama~\cite{koyama_prime_2001} for $E_\Gamma(X)$.

We prove Theorem~\ref{thm:subconvexitysym} by following the argument that appears
for $\mathbb{H}^{2}$ in~\cite[pp.~219--222]{luo_quantum_1995} with some modifications.
\begin{proof}[Proof of Theorem~\ref{thm:subconvexitysym}]
    The symmetric square $L$-function $L^{(2)}(s,u_{j})$ is an
    entire function (see~\cite{shimura_1975})
    with Gamma factors
    \[
    \gamma(s, r_{j})=\pi^{-3s}\Gamma(s)\Gamma(s+ir_{j})\Gamma(s-ir_{j}),
    \]
    and the functional equation
    \begin{equation}\label{eq:symfunc}
        L^{(2)}(s,u_{j})\gamma(s,r_{j})=L^{(2)}(1-s,u_{j})\gamma(1-s,r_{j}).
    \end{equation}
    Let $w=\tfrac{1}{2}+it_{0}$. Consider the integral
    \[I_{1}=\frac{1}{2\pi i}\int_{(\sigma)}L^{(2)}(s+w,
        u_{j})\Gamma(s+l)x^{-2s}\frac{ds}{s},\]
    where $l>4$ is an integer paramater to be chosen later,
    $\sigma=\frac{1}{2}+\frac{1}{\log r_{j}}$ and $x>0$.
    Since we are in the region of absolute convergence, we can write
    \[I_{1}=\sum_{n\in\zin}\frac{c_{j}(n)}{N(n)^{w}}F(N(n)x^{2}),\]
    where $F(u)$ is the incomplete Gamma function given by
    \[F(u)=\frac{1}{2\pi i}\int_{(\sigma)}\Gamma(s+l)u^{-s}\frac{ds}{s}.\]
    The integrand in $I_{1}$ has a simple pole at $s=0$ with residue
    $\Gamma(l)L^{(2)}(w,u_{j})$. We shift the contour
    to $(-\sigma)$ and hence by Cauchy's theorem we obtain
    \begin{equation}\label{eq:preapprox}
        I_{1}=\Gamma(l)L^{(2)}(w,u_{j})+\frac{1}{2\pi
            i}\int_{(-\sigma)}L^{(2)}(s+w,u_{j})\Gamma(s+l)x^{-2s}\frac{ds}{s}.
    \end{equation}
    Denote the integral on the right-hand side by $I_{2}$. By the functional
    equation~\eqref{eq:symfunc} we can rewrite the integral as
    \[
    I_{2}=
    \frac{-1}{2\pi i}\int_{(\sigma)}
    L^{(2)}(s+\wbar,u_{j})
    \frac{\gamma(s+\wbar,r_{j})}{\gamma(-s+w,r_{j})}
    \Gamma(-s+l)x^{2s}\frac{ds}{s}.
    \]
    We use Stirling's formula in the holomorphic
    form~\cite[8.327~(1)]{gradshteyn2007} to estimate
    the spectral Gamma factors as
    \[
    \frac{
         \Gamma(s+\wbar+ir_{j})\Gamma(s+\wbar-ir_{j})
         }{
         \Gamma(-s+\wbar+ir_{j})\Gamma(-s+\wbar-ir_{j})
         }
    =
    r_{j}^{4s}r_{j}^{-4it_{0}}\pi^{-3(2s+\wbar-w)}
    (1+O(|s+\wbar|^{4+\delta}r_{j}^{-1})).
    \]
    To estimate the integral over the error term we use the bound
    $L^{(2)}(s+\wbar,u_{j})\ll r_{j}^{\delta}$, for any $\delta>0$, which
    follows from the mean Ramanujan
    bound of Koyama~\cite[Theorem~2.1]{koyama_supnorm_1995,koyama_supnorm_2016} (notice that, due to different
    normalisation of the Hecke operators, his $\eta_{j}(n)$ is our
    $\lambda_{j}(n)$):
    \[\sum_{N(n)<N}|\lambda_{j}(n)|^{2}=O((1+|r_{j}|^{\epsilon})N).\]
    Therefore, $I_{2}$ becomes
    \[I_{2}=-r_{j}^{-4it_{0}}\pi^{6it_{0}}\sum_{n\in\zin}\frac{c_{j}(n)}{N(n)^{\wbar}}F\left(w,
            \frac{N(n)\pi^{6}}{x^{2}r_{j}^{4}}\right)+O(|w|^{5+\delta}x^{2\sigma}r_{j}^{1+\delta}),\]
    where
    \[F(w,u) = \frac{1}{2\pi
            i}\int_{(\sigma)}\Gamma(-s+l)\frac{\Gamma(s+\wbar)}{\Gamma(-s+w)}u^{-s}\frac{ds}{s}.\]
    Substituting this back into~\eqref{eq:preapprox} gives a type of
    asymmetric approximate functional equation,
    \begin{multline*}
        \Gamma(l)L^{(2)}(w,u_{j})=\sum_{n\in\zin}\frac{c_{j}(n)}{N(n)^{w}}F(N(n)x^{2})\\+
        r_{j}^{-4it_{0}}\pi^{6it_{0}}\sum_{n\in\zin}\frac{c_{j}(n)}{N(n)^{\wbar}}F\left(w,\frac{N(n)\pi^{6}}{x^{2}r_{j}^{4}}\right)
        +O(|w|^{5+\delta}x^{2\sigma}r_{j}^{1+\delta}).
    \end{multline*}
    Now integrate the above equation from $\pi^{3/2}/r_{j}$ to $e\pi^{3/2}/r_{j}$ with respect to
    the logarithmic measure $\frac{dx}{x}$. This gives
    \begin{multline*}
        \Gamma(l)L^{(2)}(w,u_{j})=\int_{\pi^{3/2}/r_{j}}^{e\pi^{3/2}/r_{j}}\sum_{n\in\zin}\frac{c_{j}(n)}{N(n)^{w}}F(N(n)x^{2})\frac{dx}{x}\\
        +r_{j}^{-4it_{0}}\pi^{6it_{0}}\int_{\pi^{3/2}/r_{j}}^{e\pi^{3/2}/r_{j}}\sum_{n\in\zin}\frac{c_{j}(n)}{N(n)^{\wbar}}F\left(w,\frac{N(n)\pi^{6}}{x^{2}r_{j}^{4}}\right)\frac{dx}{x}
        +O(|w|^{5+\delta}T^{\delta}),
    \end{multline*}
    as long as we assume that $r_{j}\sim T$. We then do the following change of
    variables: in the first integral let $x^{2}=y^{2}\pi^{3}/T^{2}$, whereas in the
    second integral we let $\pi^{3}/x^{2}r_{j}^{4}=y^{2}/T^{2}$. Thus we obtain
    \begin{multline*}
        \Gamma(l)L^{(2)}(w,u_{j})=\int_{T/r_{j}}^{eT/r_{j}}\sum_{n\in\zin}\frac{c_{j}(n)}{N(n)^{w}}F\left(N(n)\frac{y^{2}\pi^{3}}{T^{2}}\right)\frac{dy}{y}\\
        +r_{j}^{-4it_{0}}\pi^{6it_{0}}\int_{T/er_{j}}^{T/r_{j}}\sum_{n\in\zin}\frac{c_{j}(n)}{N(n)^{\wbar}}F\left(w,\frac{N(n)y^{2}\pi^{3}}{T^{2}}\right)\frac{dy}{y}
        +O(|w|^{5+\delta}T^{\delta}).
    \end{multline*}
    Squaring and summing over $r_{j}\sim T$ gives
    \begin{equation}\label{eq:splitsum}
    \begin{split}
        &\sum_{r_{j}\sim T} \left|L^{(2)}(w,u_{j})\right|^{2}
		\ll
        \int_{1/2}^{e}\sum_{r_{j}\sim
            T}\left|\sum_{n\in\zin}\frac{c_{j}(n)}{N(n)^{w}}F\left(\frac{N(n)y^{2}\pi^{3}}{T^{2}}\right)\right|^{2}\frac{dy}{y}
        \\
        &+
        \int_{1/2e}^{1}\sum_{r_{j}\sim
            T}\left|\sum_{n\in\zin}\frac{c_{j}(n)}{N(n)^{\wbar}}F\left(w,\frac{N(n)y^{2}\pi^{3}}{T^{2}}\right)\right|^{2}\frac{dy}{y}+O(|w|^{10+\delta}T^{3+\delta}).
    \end{split}
    \end{equation}
    We split the inner sums in~\eqref{eq:splitsum} at
    height $T^{2+\epsilon}$ and $(T|w|)^{2+\epsilon}$, respectively.
    Let $u>0$. For the incomplete Gamma function we use the standard estimate
    \begin{equation}\label{2411:eq001}
    F(u)\ll u^{l-1}e^{-u},
    \end{equation}
    see~\cite[8.357]{gradshteyn2007}.
    For $F(w,u)$ we shift the line of integration to $(l-1/2)$, obtaining
    $F(w,u)\ll(u/|w|^{2})^{-l+(1/2)}$, where the constant
    depends on $l$. Thus if $u\gg |w|^2T^{\epsilon}$, then by picking
    $l=\left[\frac{M}{\epsilon}\right]+4$, for some positive integer $M$, we obtain
    \[
    F(w,u)\ll T^{-M}\left(|w|^{2}/u\right)^{3/2+\delta}.
    \]
    Observe now that trivially $\lambda_{j}(n)\ll\sigma_{1}(n)/\sqrt{N(n)}$, where
    $\sigma_{1}(n)$ is as in the statement of Theorem~\ref{Kuznetsovformula}.
    It follows that $c_{j}(n)/N(n)^{\wbar}\ll N(n)^{1/2+\delta}$.
    Therefore, if we pick e.g.~$M=7$, we obtain immediately
    \begin{equation}\label{2411:eq002}
    \begin{aligned}
        \sum_{N(n)\gg
            (T|w|)^{2+\epsilon}}\frac{c_{j}(n)}{N(n)^{\wbar}}F\left(w,\frac{N(n)y^{2}\pi^{3}}{T^{2}}\right)&\ll|w|^{3+\delta},\\
        \sum_{N(n)\gg
            T^{2+\epsilon}}\frac{c_{j}(n)}{N(n)^{w}}F\left(\frac{N(n)y^{2}\pi^{3}}{T^{2}}\right)&\ll1.
    \end{aligned}
    \end{equation}
    Consider now the finite part of the sum in~\eqref{eq:splitsum}.
    In contrast to the proof in Luo and Sarnak~\cite[pg.~222]{luo_quantum_1995}, applying the best known large sieve in $\mathbb{H}^3$
	(see Watt~\cite[Theorem~1]{watt_spectral_2014}) is not helpful. Instead we proceed as follows.
	For convenience we introduce a weight function
    \[
    h(r_j) = \exp\left(-\frac{(r_j-T)^2}{M^2}\right) +
    \exp\left(-\frac{(r_j+T)^2}{M^2}\right),
    \]
    with $M=T^{1-\epsilon}$, so that the sums in $r_j$ are unchanged (by positivity) up to a factor $O(T^\epsilon)$.
    Then we open the square and recall the definition $c_j(n)=\sum_{kl^2=n} \lambda_j(k^2)$
    to expand the sums. Thus for the first line in~\eqref{eq:splitsum} we get
    \begin{equation}\label{0108:eq001}
    \begin{gathered}
    \sum_{N(l_1^2k_1)\ll T^{2+\epsilon}\atop N(l_2^2k_2)\ll T^{2+\epsilon}} \frac{1}{N(l_1^2k_1)^w N(l_2^2k_2)^{\bar w}}
		F\left(\frac{N(l_1^2k_1)y^{2}\pi^{3}}{T^{2}}\right)
		F\left(\frac{N(l_2^2k_2)y^{2}\pi^{3}}{T^{2}}\right)
		\\
		\times
		\sum_{r_j} \lambda_j(k_1^2)\lambda_j(k_2^2) h(r_j).
    \end{gathered}
    \end{equation}
	The second sum in~\eqref{eq:splitsum} can be treated similarly so we omit
    the analysis here.
	Since the cutoff function $F(x)$ is $O(1)$ in the range we are considering, we can estimate~\eqref{0108:eq001} by
	\begin{gather}
	\sum_{N(l), N(l_2)\ll T^{1+\epsilon}} \frac{1}{N(l_1l_2)}
	\sum_{N(k_1), N(k_2)\ll T^{2+\epsilon}} \frac{1}{N(k_1k_2)^{1/2}}
		\sum_{r_j} \lambda_j(k_1^2)\lambda_j(k_2^2) h(r_j)
	\nonumber
	\\
	\ll
	T^\epsilon\sum_{N(k_1), N(k_2)\ll T^{2+\epsilon}} \frac{1}{N(k_1k_2)^{1/2}}
		\Big|\sum_{r_j} \lambda_j(k_1^2)\lambda_j(k_2^2) h(r_j)\Big|, \label{0108:eq002}
	\end{gather}
	and the problem is reduced to estimating the sum over $r_j$.
	For this we apply the Kuznetsov formula, so that~\eqref{0108:eq002} is transformed into a sum of Kloosterman sums, namely
	\[
	\sum_{N(k_1),N(k_2)\ll T^{2+\epsilon}} \frac{1}{N(k_1k_2)^{1/2}}
    \left\lvert\sum_{c\in\mathbb{Z}\backslash\{0\}} \frac{S(\bar{k}_1^2,\bar{k}_2^2,c)}{N(c)} \breve{h}\left(\frac{2\pi
    k_1k_2}{c}\right)\right\rvert,
	\]
    where $\breve{h}(z)$ is the Bessel transform (as given in the $S$-term in
    Theorem~\ref{Kuznetsovformula}).
	It is not difficult to show that the continuous contribution and the term associated to $\delta_{m=\pm n}$ appearing in the Kuznetsov formula
	are bounded by $O(T^{3+\epsilon})$.
	To treat the Kloosterman sums we apply Weil's bound~\eqref{s02:weil}, and we estimate carefully the function $\breve{h}(z)$.
    As a first step we use an integral representation for the function $\breve{h}(z)$ (see~\cite[(2.10)]{motohashi_trace_1997}), writing
	\[
	\breve{h}(z) = \frac{4i}{\pi^2} \int_0^{\pi/2}\cos(2|z|\cos\vartheta\sin\tau)I(2|z|\cos\tau)d\tau,
	\]
	where $z=|z|e^{i\vartheta}$ and
	\begin{equation}\label{Idef}
	I(x) = 	\int_{-\infty}^{\infty} r^2 h(r) \cosh(\pi r) K_{2ir}(x) dr.
	\end{equation}
	The analysis of the integral $I(x)$ can be performed by following the work
    of Li (see~\cite[\S5]{li_bounds_2011}).
	First we have, for $A>1$,
	\[
	I(x) \ll \exp(x) T^{3+\epsilon} (x/T)^A,
	\]
	which implies that in the range where $x\ll 1$, we can estimate
	\[
	\sum_{N(k_1),N(k_2)\ll T^{2+\epsilon}} N(k_1k_2)^{\epsilon}
    \sum_{N(c)\gg N(k_1k_2)} \frac{d(c)|(\bar{k}_1^2,\bar{k}_2^2,c)|}{N(c)^{1+\epsilon}} \; T^{-A} \ll 1.
	\]
	This shows that the tail of the sum in $c$ is negligible.
	For the rest of the sum we evaluate $I(x)$ by showing that the total mass
    comes from a neighbourhood of $|x|=T$.
	More precisely, we have
	\begin{equation}\label{Iestimate1}
	I(x) \ll T^{-A}, \qquad\text{for}\; |x|\leq \frac{T}{100}\;\text{or}\; |x|\geq 100 T,
	\end{equation}
	and
	\begin{equation}\label{Iestimate2}
	I(x) = L \, x^2 \exp\left(-\frac{(x-2T)^2}{4M^2}\right) + O\left(\frac{T^2x}{M^3}\right)\quad\text{otherwise},
	\end{equation}
	for some absolute constant $L$. Note that~\eqref{Iestimate1} and~\eqref{Iestimate2}
	can be proved almost verbatim as in~\cite[Proposition~5.1]{li_bounds_2011}.
	The error in~\eqref{Iestimate2} contributes,
	recalling~\eqref{0608:eq001}, at most $O(T^{3+\epsilon})$.
	By using~\eqref{Iestimate1} we can further reduce the sum over $c$: if
    $N(c)\gg N(k_1k_2)/T^2$ then we obtain again
	a contribution of $O(1)$.
	We therefore need to prove the estimate
	\begin{equation}\label{0108:eq003}
	\sum_{N(k_1)\ll T^{2+\epsilon}\atop N(k_2)\ll T^{2+\epsilon}} \frac{1}{N(k_1k_2)^{1/2}}
		\!\! \sum_{N(c)\ll\frac{N(k_1k_2)}{T^2}}
        \frac{d(c)\bigl\lvert(\bar{k}_1^2,\bar{k}_2^2,c)\bigr\rvert}{N(c)^{1/2}}
		\Big|\mathring{h}\left(\frac{2\pi k_1k_2}{c}\right)\Big|
	\ll
	T^{4+\epsilon},
	\end{equation}
	where $\mathring{h}(z)$ is the function
	\[
	\mathring{h}(z) = \int_0^{\pi/2}\cos(2|z|\cos\vartheta\sin\tau)g(\cos\tau)G(|z|\cos\tau) d\tau,
	\]
    with
	\[
	G(x) = x^2 \exp\left(-\frac{(x-T)^2}{M^2}\right),
	\]
	and $g(x)$ is a smooth characteristic function of the interval $[c_1T/|z|,c_2T/|z|]$ (for some constants $0<c_1<1<c_2$),
	such that $g^{(j)}(x) \ll (T/|z|)^{-j}$.
	We claim that
	\begin{equation}\label{h-estimate}
	\mathring{h}(z)
	\ll
	\begin{cases}
        T^3|z|^{-1}, & \text{if }|z|\gg T,\\
    T^{1+\epsilon} |\cos\vartheta|^{-1}, &\text{if } T\ll|z|\ll T^{1+\epsilon}
    \text{ and } |\cos\vartheta|\gg T^{-1+3\epsilon},
	\\
    T^{-A}, & \text{if }|z|\gg T^{1+\epsilon} \text{ and } |\cos\vartheta|\gg
    |z| T^{-2+3\epsilon}.
	\end{cases}
	\end{equation}
	The first estimate follows by bounding in absolute value $g(\cos\tau)G(|z|\cos\tau)\ll T^2$,
	and the fact that the support of $g$ is $[c_1T/|z|,c_2T/|z|]$.
	Integrating by parts once in $\tau$ (and noting that the exponential is piecewise monotonic,
	so that the derivative is piecewise of constant sign) leads to the second bound.
	Finally, to prove the third bound we integrate by parts multiple times. The first integration gives
	\[
	\mathring{h}(z)
	=
	-\int_0^{\pi/2} \sin(2|z|\cos\vartheta\sin\tau) \frac{d}{d\tau}\left(\frac{g(\cos\tau)G(|z|\cos\tau)}{2|z|\cos\vartheta\cos\tau}\right) d\tau.
	\]
	Notice that $G'(x)\approx G(x)/x$ and that $g$ vanishes at the endpoints.
    After expanding the derivative, we see that the largest contribution
	comes from the integral
	\[
	\int_0^{\pi/2} \sin(2|z|\cos\vartheta\sin\tau) g(\cos\tau) \frac{G(|z|\cos\tau)\sin\tau}{2|z|\cos\vartheta\cos^2\tau} d\tau.
	\]
	Integrating by parts $k$-times we thus obtain
	\[
	\mathring{h}(z)\ll
	T^2 \int_0^{\pi/2} g(\cos\tau) \left( \frac{\sin\tau}{|z|\cos\vartheta\cos^2\tau} \right)^k d\tau
	\\
	\ll
	T^2 \left( \frac{|z|}{T^2\cos\vartheta}	\right)^k
	\ll
	T^{-A}.
	\]
	This proves the last bound in~\eqref{h-estimate}.

	Let us now return to proving the estimate~\eqref{0108:eq003}.
    To treat the sum of gcd's, we proceed as follows.
	For $u\in\mathbb{Z}[i]$ write $u=u_1u_2^2$ with $u_1$ squarefree, and complete the square, defining $u_+=u_1^2u_2^2$.
	Note in particular that $|u_+|\geq |u|$.
	We can bound the left-hand side in~\eqref{0108:eq003}, by splitting the sum
    into congruence classes, by
    \begin{equation}\label{eq:congclasses}
        \sum_{N(u) \ll T^{2+\epsilon}} |u|
        \sum_{k_1^2, k_2^2\equiv 0\;(u)} \frac{1}{N(k_1k_2)^{1/2}}
        \sum_{c\equiv 0\;(u)} \frac{1}{N(c)^{1/2-\epsilon}} \Big|\mathring{h}\left(\frac{2\pi k_1k_2}{c}\right)\Big|,
    \end{equation}
	where we use the fact that the variables of summation, $k_1,k_2,c$, are
    restricted to the ranges
	\[
	N(k_1),N(k_2)\ll T^{2+\epsilon}, \quad N(c) \ll \frac{N(k_1k_2)}{T^2}.
	\]
	After relaxing the condition $k_j^2\equiv 0$ (mod $u$) to $k_j\equiv 0$ (mod $\sqrt{u_+}$),
	and changing variables $k_j\mapsto\sqrt{u_+}k_j$ and $c\mapsto uc$,
    we can bound~\eqref{eq:congclasses} by
	\[
	\sum_{N(u) \ll T^{2+\epsilon}} \frac{1}{|u_+|}
		\sum_{ N(k_j) \ll \frac{ T^{2+\epsilon} }{ N(u_+)^{1/2} } } \!\!\! \frac{1}{N(k_1k_2)^{1/2}}
			\!\!\!\!\!\! \sum_{N(c)\ll \frac{N(u_+k_1k_2)}{T^2N(u)}} \frac{1}{N(c)^{1/2-\epsilon}} \Big|\mathring{h}\left(\frac{2\pi u_+k_1k_2}{uc}\right)\Big|.
	\]
	We perform a dyadic partition of unity for $k_1,k_2$, and $c$, so that we can restrict our attention
	to sums with
	\[
	|k_1|\sim K_1,\quad |k_2|\sim K_2,\quad |c|\sim Q.
	\]
	We distinguish three cases according to whether the quantity $|u_+K_1K_2/uQ|$ is in the interval $[T,T^{1+\epsilon}]$,
	$[T^{1+\epsilon},T^{3/2+\epsilon}]$, or $[T^{3/2+\epsilon},T^{2+\epsilon}]$, in order to exploit the different bounds in~\eqref{h-estimate}.
	Solving for $Q$, we set $Q_*=K_1K_2|u_+|/T|u|$, and we define
	\[
	Q_1 = Q_*,
	\quad
	Q_2 = Q_* T^{-\epsilon},
	\quad
	Q_3 = Q_* T^{-1/2-\epsilon},
	\quad
	Q_4 = Q_* T^{-1-\epsilon},
	\]
	so that we need to discuss separately $Q_{s+1}\ll Q\ll Q_s$, for $s=1,2,3$.
	Let us consider the first case, $Q_2\ll Q\ll Q_1$. To estimate
    $|\mathring{h}(2\pi u_+k_1k_2/uc)|$, we apply
	the second bound in~\eqref{h-estimate} if $|\cos\vartheta|\gg T^{-1+3\epsilon}$, and the first bound otherwise.
	Hence we obtain
	\begin{equation}\label{1caseestimate1}
	\begin{split}
	\sum_{|u|\leq T^{1+\epsilon}}\frac{|u|^\epsilon}{|u_+|}
	&	\sum_{K_j\leq \frac{T^{1+\epsilon}}{\sqrt{|u_+|}}}
			\sum_{Q_2\ll Q\ll Q_1} \frac{1}{K_1K_2Q}
	\\
	&\times
	\sum_{|k_j|\sim K_j}
		\bigg(
		\sum_{\substack{|c|\sim Q\\|\cos\vartheta|\gg T^{-1+3\epsilon}}}\frac{T^{1+\epsilon}}{|\cos\vartheta|}+
		\sum_{\substack{|c|\sim Q\\|\cos\vartheta|\ll T^{-1+3\epsilon}}}T^2
		\bigg).
	\end{split}
	\end{equation}
	The first summand in parentheses can be bounded by
	\[
	\sum_{\substack{|q|\sim Q\\|\cos\vartheta|\gg T^{-1+3\epsilon}}}\frac{T^{1+\epsilon}}{|\cos\vartheta|}
	\ll
	\int_{Q}^{2Q} r
		\int_{\hspace{-14pt}\raisebox{-10pt}{\ensuremath{\vartheta\in[0,2\pi]:\atop|\cos\vartheta|\gg T^{-1+3\epsilon}}}}
			\frac{T^{1+\epsilon}}{|\cos\vartheta|}d\vartheta dr
	\ll
	Q^2T^{1+\epsilon}.
	\]
	Analogously, the second summand can be bounded by
	\[
	\sum_{\substack{|c|\sim Q\atop |\cos\vartheta|\ll T^{-1+3\epsilon}}}T^2\ll
	T^2\left(Q^2T^{-1+3\epsilon}+Q\right),
	\]
    which corresponds to counting the number of integral points in a sector (on
    a plane)
    of width $T^{-1+3\epsilon}$
	(and can be estimated by the area of the sector plus the length of its boundary).
	Combining the two estimates above we deduce that \eqref{1caseestimate1}~can be bounded by
	\begin{equation}\label{1casefinalestimate}
	\begin{split}
	&\sum_{|u|\leq T^{1+\epsilon}}\frac{|u|^\epsilon}{|u_+|}
		\sum_{K_j\leq \frac{T^{1+\epsilon}}{\sqrt{|u_+|}}}
			\sum_{Q_2\ll Q\ll Q_1} K_1 K_2 (QT^{1+3\epsilon}+T^2)
	\\
	&\ll
	\!\!\!\! \sum_{|u|\leq T^{1+\epsilon}} \! \frac{|u|^\epsilon}{|u_+|} \!
		\sum_{K_j\leq \frac{T^{1+\epsilon}}{\sqrt{|u_+|}}} \!\!\!\! K_1K_2 (Q_1T^{1+3\epsilon} + T^2)
	\ll
	\!\!\!\! \sum_{|u_+|\leq T^{1+\epsilon}} \! \frac{T^{4+10\epsilon}}{|u_+|^2}
	\ll
	T^{4+11\epsilon}.
	\end{split}
	\end{equation}
	Let us consider now the second case, $Q_3\ll Q\ll Q_2$. The portion of the
    sum, where
	$|\cos\vartheta|\gg Q_*/QT^{1-3\epsilon}$, is negligible in view of the last bound in~\eqref{h-estimate}.
	If, on the other hand, $|\cos\vartheta|\ll Q_*/QT^{1-3\epsilon}$, then we apply the first bound in~\eqref{h-estimate},
	and we get
	\begin{equation}\label{2caseestimate1}
	\begin{split}
	\sum_{0<|u|\le T^{1+\epsilon}} \frac{|u|^\epsilon}{|u_+|}
		\sum_{K_j\leq \frac{T^{1+\epsilon}}{\sqrt{|u_{+}|}}}
			\sum_{Q_3\ll Q\ll Q_2} \frac{1}{K_1K_2Q}
				\sum_{|k_j|\sim K_j}
					\sum_{|c|\sim Q\atop |\cos\vartheta|\ll\frac{Q_*}{QT^{1-3\epsilon}} }
					\frac{QT^2}{Q_*}.
	\end{split}
	\end{equation}
	Arguing in the same way as in~\eqref{1caseestimate1}--\eqref{1casefinalestimate}
	to estimate the inner sum, we obtain $O(T^{4+\epsilon})$ also in this case.
	Finally, we consider the case $Q_4\ll Q\ll Q_3$. We apply the first bound in~\eqref{h-estimate},
	and we obtain the sum
	\begin{equation}\label{3caseestimate1}
	\begin{split}
	\sum_{0<|u|\le T^{1+\epsilon}}
	&\frac{|u|^\epsilon}{|u_+|}
		\sum_{K_j\leq \frac{T^{1+\epsilon}}{\sqrt{|u_{+}|}}}
			\sum_{Q_4\ll Q\ll Q_3} \frac{1}{K_1K_2Q}
				\sum_{|k_j|\sim K_j}
					\sum_{|c|\sim Q} \frac{QT^2}{Q_*}
	\\
	&\ll
	\sum_{0<|u|\le T^{1+\epsilon}} \frac{|u|^{1+\epsilon}}{|u_+|^2}
		\sum_{K_j\leq \frac{T^{1+\epsilon}}{\sqrt{|u_+|}}}
			\sum_{Q_4\ll Q\ll Q_3} T^3Q^2
	\\
	&\phantom{xxxxxxxx}\ll
	\sum_{0<|u|\le T^{1+\epsilon}}\frac{|u|^{1+\epsilon}}{|u_+|^2}
		\sum_{K_j\leq \frac{T^{1+\epsilon}}{\sqrt{|u_+|}}}
		\frac{K_1^2K_2^2|u_+|^2}{T^{2\epsilon}|u|^2}
	\ll
	T^{4+\epsilon}.
	\end{split}
	\end{equation}
	This proves~\eqref{0108:eq003}, and concludes the proof of Theorem~\ref{thm:subconvexitysym}.
\end{proof}

We  can now prove Theorem~\ref{1610:thm01}.
Our proof follows closely the method in~\cite{koyama_prime_2001}
(see also~\cite[\S8]{iwaniec_prime_1984}, \cite[\S6]{luo_quantum_1995}).

\begin{proof}[Proof of Theorem~\ref{1610:thm01}]
    Let $f$ be a smooth function supported in $[\sqrt{N}, \sqrt{2N}]$
    satisfying $|f^{(p)}(\xi)| \ll N^{-p/2}$ for $p\geq 0$ and
    \begin{eqnarray*}
        \int_{0}^{\infty} f(\xi) \xi\, d \xi =N.
    \end{eqnarray*}
    Then, applying Corollary~\ref{cor:subconvexity} and following the proof
    of~\cite[Lemma~4.3]{koyama_prime_2001} we deduce that
    \begin{eqnarray}
        \sum_{n \in\zin} \frac{r_j f(|n|) |\rho_j(n)|^2}{\sinh \pi r_j} = c N +
        \mathfrak{r}(r_j, N),
    \end{eqnarray}
    where $c>0$ is some constant and
    \begin{eqnarray}
        \sum_{|r_j| \leq T} |\mathfrak{r}(r_j, N)| \ll T^{7/2+\epsilon} N^{1/2}.
    \end{eqnarray}
    We conclude that
    \begin{multline}\label{eq:specsum}
        \frac{1}{N} \sum_{n \in\zin} f(|n|) \sum_{|r_j| \leq T}  \frac{r_j |\rho_j(n)|^2}{\sinh \pi r_j} X^{ir_j} \exp(-r_j/T) \\
        = c  \sum_{|r_j| \leq T}  X^{ir_j} \exp(-r_j/T) + O \left(T^{7/2+\epsilon} N^{-1/2}\right).
    \end{multline}
    We apply the Kuznetsov formula (Theorem~\ref{Kuznetsovformula}) to the inner
    spectral sum on the left-hand side with the test function
    \[h(r) = \frac{\sinh (\pi + 2i \alpha)r}{\sinh \pi r},\]
    where $2 \alpha = \log X +i/T$. For $r>0$, we get
    \[h(r) = X^{ir}e^{-r/T} + O(e^{-\pi r}).\]
    The contribution of each term ($D, C, U$ and $S$) in the Kuznetsov formula can be
    computed separately. For the test function $h(r)$ chosen above,
    Koyama~\cite[p.~789]{koyama_prime_2001} proves that
    \[ |C|+|U|\ll T^{2}.\]
    For the sum in $S$, we apply estimates for Bessel functions and Weil's bound
    \eqref{s02:weil} (see~\cite[pp.~789--792]{koyama_prime_2001} for details).
    Following~\cite[pp.~790--791]{koyama_prime_2001} we obtain
    \begin{eqnarray}
        S \ll N^{1/2+ \epsilon} T^{1/2+\epsilon} X^{1/2}.
    \end{eqnarray}
    Hence~\eqref{eq:specsum} gives
    \[\sum_{r_j} X^{ir_j}e^{-r_j/T} \ll N^{1/2+ \epsilon} T^{1/2+\epsilon}
        X^{1/2} + T^{7/2+ \epsilon} N^{-1/2},\]
    and choosing $N=X^{-1/2} T^{3}$ we get
    \begin{equation}\label{eq:stxresult}
        \sum_{r_j} X^{ir_j}e^{-r_j/T} \ll T^{2+\epsilon} X^{1/4}.
    \end{equation}
    The result for the sharp sum, $S(T,X)$, is obtained by a standard approximation argument.
\end{proof}
Theorem~\ref{intro:thm01} then follows immediately.
\begin{proof}[Proof of Theorem~\ref{intro:thm01}]
    From Theorem~\ref{1610:thm01} and summation by parts we have
    \[
	\sum_{|r_j| \leq T} \frac{X^{1+ir_j}}{1+ir_j} \ll T^{1+\epsilon} X^{5/4}.
	\]
    Therefore, by the explicit formula~\eqref{1610:eq001}, we conclude that
    \[
	E_{\Gamma}(X) = O \left( X^2 T^{-1} \log^2 X + T^{1+\epsilon} X^{5/4}\right).
	\]
    Choosing $T = X^{3/8}$, we finally obtain the bound $O(X^{13/8+\epsilon})$.
\end{proof}

\section{Second moment}\label{S04}

In this section we prove Theorem~\ref{intro:thm02} by
applying
the Selberg trace formula for a suitably chosen test function.
Let $q(x)$ be a smooth, even, non-negative function on $\reals$,
with compact support contained in $[-1,1]$ and of unit mass (i.e.~$\|q\|_1=1$).
Let $X>1$ and $0<\delta<1/4$, and consider the functions
\[
\begin{gathered}
g_s(x) = 4\left(\sinh\frac{x}{2}\right)^2\mathbf{1}_{[0,s]}(|x|),\quad s=\log X,
\\
g_\pm(x) = (g_{s\pm\delta}\ast q_\delta)(x),
\quad
q_\delta(x)=\delta^{-1}q(x/\delta).
\end{gathered}
\]
We claim that $\psi_\Gamma(X)$ is obtained, up to an error bounded by $O(X^{3/2})$,
by summing $g_s$ over hyperbolic and loxodromic conjugacy classes, with appropriate weights.
More precisely, we claim that
\begin{equation}\label{2609:eq005}
\psi_\Gamma(X) =
\sum_{\{T\}hyp+lox} \frac{g_s(\log N(T))\Lambda_{\Gamma}(
    N(T))}{|\mathcal{E}(T)||a(T)-a(T)^{-1}|^2} + O(X^{3/2})
\end{equation}
as $X\to\infty$.
To see this, observe that
\[
\psi_\Gamma(X) =
\sum_{\{T\}hyp+lox} \frac{g_s(\log N(T))\Lambda_{\Gamma}(N(T))}{|N(T)^{1/2}-N(T)^{-1/2}|^2},
\]
and notice that, for $z\in\mathbb{C}$,
\[
\frac{1}{||z|-|z|^{-1}|^2} = \frac{1}{|z-z^{-1}|^2} + O\left(\frac{1}{|z|^3}\right)\quad\text{as }|z|\to\infty.
\]
Moreover, recall from Section~\ref{S02} that
$|\mathcal{E}(T)|\neq 1$ only for finitely many $T$. We can therefore write
\[
\psi_\Gamma(X) =
\sum_{\{T\}hyp+lox} \frac{g_s(\log N(T))\Lambda_{\Gamma}(
    N(T))}{|\mathcal{E}(T)||a(T)-a(T)^{-1}|^2}
+
O\left(
1 +
\sum_{N(T)\leq X} \frac{\Lambda_{\Gamma}(N(T))}{N(T)^{1/2}}
\right).
\]
Using the bound $\psi_\Gamma(X)\ll X^2$ and summation by parts on the error, we get~\eqref{2609:eq005}.
The functions $g_\pm$ give smooth versions of the sum in~\eqref{2609:eq005}, which we denote by $\psi_\pm$.
By definition of $g_\pm$ we have, for $x\geq 0$,
\[
g_{-}(x) + O(\delta e^x\mathbf{1}_{[0,s]}(x))
\leq
g_s(x)
\leq
g_{+}(x) + O(\delta e^x\mathbf{1}_{[0,s]}(x)),
\]
and using again $\psi_\Gamma(X)\ll X^2$ this implies in turn:
\begin{equation}\label{2609:eq004}
\psi_{-}(X) + O(X^{3/2}+\delta X^2)
\leq
\psi_\Gamma(X)
\leq
\psi_{+}(X) + O(X^{3/2}+\delta X^2).
\end{equation}
Let $h_s(r)$
be the Fourier transform of $g_s$ (normalized as in~\eqref{2909:eq001}),
and $h_\pm=\widehat{g}_\pm=h_{s\pm\delta}\widehat{q}_\delta$.
We apply the Selberg trace formula (Theorem~\ref{s02:selberg-trace-formula})
to the pair $(g_\pm,h_\pm)$.
For simplicity we assume that the group has only one cusp at infinity,
although the same proof works for multiple cusps. We obtain
\[
\psi_\pm(X)
=
\sum_{j}h_\pm(r_{j})
-
\frac{1}{4\pi}\int_{-\infty}^{\infty}h_\pm(r)\frac{\varphi'}{\varphi}(1+ir)dr
- I - E  - P + O(1),
\]
where the terms $I$, $E$, $P$ are given in Theorem~\ref{s02:selberg-trace-formula}.
It is easy to prove, by definition of $h_\pm$
and the fact that
$\widehat{q}_\delta^{(j)}(r)\ll_{j, k}\delta^j(1+|\delta r|)^{-k}$,
together with standard estimates on hyperbolic functions, that the terms $I$, $E$, and $P$ contribute at most $O(X)$.
The eigenvalues $\lambda_j\in[0,1]$ contribute
\begin{equation}\label{2709:eq001}
M_\pm(X) := \sum_{\lambda_j\in[0,1]} h_\pm(r_j) = M(X) + O(\delta X^2 + X),
\end{equation}
and we are left to analyze the eigenvalues $\lambda_j>1$ and the continuous spectrum.
Before doing so, we prove a lemma that shows how to exploit the oscillation
of the function $h_\pm$ when doing integration over the variable $X$.

\begin{lemma}\label{selberg:lemma001}
Let $V,\Delta\in\reals$ with $1<\Delta\leq V$, and let
$r_1,r_2\in\reals$. We have
\[
\frac{1}{\Delta}\int_V^{V+\Delta}
h_{s\pm\delta}(r_1)h_{s\pm\delta}(r_2)dX
\ll
\frac{V^3}{\Delta}u(r_1)u(r_2)u(|r_1|-|r_2|),
\]
with $u(r)=(1+|r|)^{-1}$. The implied constant does not depend on $\delta$.
\end{lemma}
\begin{proof}
Recall that $s=\log X$.
By definition of $g_s$ we can write, for $r\in\reals$,
\[
h_s(r)
=\frac{2\sinh(s(1+ir))}{1+ir}+\frac{2\sinh(s(1-ir))}{1-ir}-\frac{4\sin(sr)}{r}.
\]
Taking $h_{s\pm\delta}(r_1)$ against $h_{s\pm\delta}(r_2)$
and integrating in $X$, the desired inequality follows by either bounding in absolute
value the integrand or by integrating by parts first.
\end{proof}

We can now estimate the contribution of the discrete and continuous spectrum.
For the discrete spectrum $\lambda_j=1+r_j^2$ with $r_j>0$, we have
\[
\frac{1}{\Delta}\int_V^{V+\Delta} \Big|\sum_{r_j>0} h_\pm(r_j)\Big|^2 dX
=
\sum_{r_j,r_\ell>0} \widehat{q}_\delta(r_j)\widehat{q}_\delta(r_\ell)
\frac{1}{\Delta}\int_V^{V+\Delta} h_{s\pm\delta}(r_j)h_{s\pm\delta}(r_\ell)dX,
\]
and we bound the integral using Lemma~\ref{selberg:lemma001}, obtaining the estimate
(up to constants)
\[
\frac{V^3}{\Delta}
\sum_{r_j,r_\ell>0} |\widehat{q}_\delta(r_j)\widehat{q}_\delta(r_\ell)|
u(r_j)u(r_\ell)u(|r_j|-|r_\ell|).
\]
The double sum can be estimated by doing a unit interval decomposition and using~\eqref{weyl-law}, obtaining
\begin{equation}\label{2709:eq002}
O\left(\frac{V^3}{\Delta\delta^3}\log(\delta^{-1})\right).
\end{equation}
The analysis of the continuous spectrum is similar, and contributes a quantity not bigger than~\eqref{2709:eq002}.
Combining~\eqref{2609:eq004}, \eqref{2709:eq001} and~\eqref{2709:eq002}, we get
\[
\begin{split}
\frac{1}{\Delta}\int_V^{V+\Delta}
|\psi_\Gamma(X)
&
-M(X)|^2 dX
\\
&\leq
\frac{1}{\Delta}\int_V^{V+\Delta} |\psi_\pm(X)-M_\pm(X)|^2 dX + O(\delta^2 V^4 + V^3)
\\
&\ll
\delta^2 V^4 + V^3 + \frac{V^3}{\Delta \delta^3} \log(\delta^{-1}).
\end{split}
\]
Choosing $\delta=(V\Delta)^{-\frac 15}\log^{\frac 15}(V\Delta)$ we obtain the theorem.

\bibliographystyle{amsplain}
\providecommand{\bysame}{\leavevmode\hbox to3em{\hrulefill}\thinspace}
\providecommand{\MR}{\relax\ifhmode\unskip\space\fi MR }
\providecommand{\MRhref}[2]{%
  \href{http://www.ams.org/mathscinet-getitem?mr=#1}{#2}
}
\providecommand{\href}[2]{#2}

\end{document}